\documentclass[11pt]{elsarticle}

\usepackage{lipsum}
\makeatletter
\def\ps@pprintTitle{%
 \let\@oddhead\@empty
 \let\@evenhead\@empty
 \def\@oddfoot{}%
 \let\@evenfoot\@oddfoot}
\makeatother

\usepackage{graphicx}
\usepackage{amsfonts}
\usepackage{amssymb}
\usepackage{amsmath}
\usepackage{amsthm}
\usepackage{hyperref}
\usepackage{tikz}
\usepackage{multirow}
\usepackage{enumerate}
\usepackage{setspace}
\usepackage[margin=1in]{geometry}

\newtheorem{theorem}{Theorem}

\newtheorem{corollary}{Corollary}
\newtheorem{proposition}{Proposition}
\newtheorem{remark}{Remark}
\newtheorem{example}{Example}
\newtheorem{definition}{Definition}

\newcommand{\R}{\mathbb{R}}

\newcommand{\sub}{\subseteq}

\newcommand{\p}{\textrm{P}}

\newcommand{\E}{\mb E}

\newcommand{\st}{\operatorname{s.t.}}
\newcommand{\rank}{\operatorname{rank}}

\newcommand{\mb}[1]{\mathbf{#1}}

\newcommand{\bs}{\boldsymbol}

\newcommand{\U}{\mathcal{U}}

\newcommand{\bbeta}{\bs\beta}
\newcommand{\bDelta}{\bs\Delta}

\newcommand{\bb}{{\bs\beta}}
\newcommand{\y}{\mb y}
\newcommand{\X}{{\mb X}}
\newcommand{\D}{\bs\Delta}

\newcommand{\hd}{\widehat{\D}}

\newcommand{\xx}{\mb x}

\newcommand{\Y}{\mb Y}
\newcommand{\uu}{\mb u}
\newcommand{\vv}{\mb v}
\newcommand{\ww}{\mb w}
\newcommand{\z}{\mb Z}

\newcommand{\zz}{\mb z}
\newcommand{\ind}{ }
\newcommand{\bd}{\bs\delta}

\newcommand{\ul}{\underline{\lambda}}
\newcommand{\ol}{\overline{\lambda}}
\newcommand{\oh}{\overline{h}}
\newcommand{\argmax}{\operatorname{argmax}}
\newcommand{\argmin}{\operatorname{argmin}}

\newcommand{\ds}{\displaystyle}
\newcommand{\gam}{\bs\gamma}
\newcommand{\sign}{\operatorname{sign}}
\newcommand{\tzz}{\tilde \zz}
\newcommand{\bzz}{\bar \zz}

\begin{document}

\begin{frontmatter}

\title{Characterization of the equivalence of robustification and regularization in linear and matrix regression}
\tnotetext[test]{Copenhaver is partially supported by the Department of Defense, Office of Naval Research, through the National Defense Science and Engineering Graduate Fellowship.}

\author{Dimitris Bertsimas}
\address{Sloan School of Management and Operations Research Center, MIT}
\ead{dbertsim@mit.edu}

\author{Martin S. Copenhaver\corref{}}

\address{Operations Research Center, MIT}
\ead{mcopen@mit.edu}

\begin{abstract}

The notion of developing statistical methods in machine learning which are robust to adversarial perturbations in the underlying data has been the subject of increasing interest in recent years. A common feature of this work is that the adversarial \emph{robustification} often corresponds exactly to regularization methods which appear as a loss function plus a  penalty. In this paper we deepen and extend the understanding of the connection between robustification and regularization (as achieved by penalization) in regression problems. Specifically,
\begin{enumerate}[(a)]
\item In the context of linear regression, we characterize precisely under which conditions on the model of uncertainty used and on the loss function penalties robustification and regularization are equivalent.

\item We extend the characterization of robustification and regularization to matrix regression problems (matrix completion and Principal Component Analysis).
\end{enumerate}

\end{abstract}


\end{frontmatter}


\section{Introduction}

 The development of predictive methods that perform well in the face of uncertainty is at the core of modern machine learning and statistical practice. Indeed, the notion of \emph{regularization}---loosely speaking, a means of controlling the ability of a statistical model to generalize to new settings by trading off with the model's complexity---is at the very heart of such work \cite{hastie}. Corresponding regularized statistical methods, such as the Lasso for linear regression \cite{tibshirani} and nuclear-norm-based approaches to matrix completion \cite{pablo,rechtmc}, are now ubiquitous and have seen widespread success in practice.

In parallel to the development of such regularization methods, it has been shown in the field of robust optimization that under certain conditions these regularized problems result from the need to immunize the statistical problem against adversarial perturbations in the data \cite{ghaoui,xu,RObook,xuChap}. Such a \emph{robustification} offers a different perspective on  regularization methods by identifying which adversarial perturbations the model is protected against. Conversely, this can help to inform statistical modeling decisions by identifying potential choices of regularizers. 
Further, this connection between regularization and robustification offers the potential to use sophisticated data-driven methods in robust optimization \cite{ddunc1,ddunc2} to design regularizers in a principled fashion.

With the continuing growth of the adversarial viewpoint in machine learning (e.g. the advent of new deep learning methodologies such as generative adversarial networks \cite{gan1,gan2,gan3}), it is becoming increasingly important to better understand the connection between robustification and regularization. Our goal in this paper is to shed new light on this relationship by focusing in particular on linear and matrix regression problems. Specifically, our contributions include:

\begin{enumerate}
\item In the context of linear regression we demonstrate that in general such a robustification procedure is not equivalent to regularization (via penalization). We characterize precisely under which conditions on the model of uncertainty used and on the loss function penalties one has that robustification is equivalent to regularization. 

\item We break new ground by considering problems in the matrix setting, such as matrix completion and Principal Component Analysis (PCA). We show that the nuclear norm, a popular penalty function used throughout this setting, arises directly through robustification. As with the case of vector regression, we characterize under which conditions on the model of uncertainty there is equivalence of robustification and regularization in the matrix setting.
\end{enumerate}

The structure of the paper is as follows. In Section 2, we review background on norms and consider robustification and regularization in the context of linear regression, focusing both on their equivalence and non-equivalence. In Section 3, we turn our attention to regression with underlying matrix variables, considering in depth both matrix completion and PCA. In Section 4, we include some concluding remarks.


\section{A robust perspective of linear regression}

\subsection{Norms and their duals}

In this section, we introduce the necessary background on norms which we will use to address the equivalence of robustification and regularization in the context of linear regression. Given a vector space $V\sub\R^n$ we say that $\|\cdot\|:V\to\R$ is a \emph{norm} if for all $\vv,\ww\in V$ and $\alpha\in\R$
\begin{enumerate}
\item If $\|\vv\|=0$, then $\vv=0$,
\item $\|\alpha \vv\|=|\alpha|\|\vv\|$ (absolute homogeneity), and
\item $\|\vv+\ww\|\leq \|\vv\|+\|\ww\|$ (triangle inequality).
\end{enumerate} 
If $\|\cdot\|$ satisfies conditions 2 and 3, but not 1, we call it a \emph{seminorm}. For a norm $\|\cdot\|$ on $\R^n$ we define its dual, denoted $\|\cdot\|_*$, to be
$$\|\bb\|_* := \max_{\xx\in\R^n} \frac{\xx'\bb}{\|\xx\|},$$
where $\xx'$ denotes the transpose of $\xx$ (and therefore $\xx'\bb$ is the usual inner product). For example, the $\ell_p$ norms $\|\bb\|_p := \left(\sum_i |\beta_i|^p\right)^{1/p}$ for $p\in[1,\infty)$ and $\|\bb\|_\infty := \max_i |\beta_i|$ satisfy a well-known duality relation: $\ell_{p^*}$ is dual to $\ell_{p}$, where $p^*\in[1,\infty]$ with $1/p+1/p^*=1$. We call $p^*$ the \emph{conjugate} of $p$. More generally for matrix norms\footnote{We treat a matrix norm as any norm on $\R^{m\times n}$ which satisfies the three conditions of a usual vector norm, although some authors reserve the term ``matrix norm'' for a norm on $\R^{m\times n}$ which also satisfies a submultiplicativity condition (see \cite[pg. 341]{hornjohnson}).} $\|\cdot\|$ on $\R^{m\times n}$ the dual is defined analogously:
$$\|\D\|_* := \max_{\mb A\in\R^{m\times n}}\frac{\langle \mb A,\D\rangle}{\|\mb A\|},$$
where $\D\in\R^{m\times n}$ and $\langle\cdot,\cdot\rangle$ denotes the trace inner product: $\langle \mb A,\D\rangle = \operatorname{Tr}(\mb A'\D)$, where $\mb A'$ denotes the transpose of $\mb A$. We note that the dual of the dual norm is the original norm \cite{boydvandenberghe}.

Three widely used choices for matrix norms (see \cite{hornjohnson}) are Frobenius, spectral, and induced norms. The definitions for these norms are given below for $\D\in\R^{m\times n}$ and summarized in  Table \ref{tab:norms} for convenient reference.

\begin{enumerate}
\item The $p$-Frobenius norm, denoted $\|\cdot\|_{F_p}$, is the entrywise $\ell_p$ norm on the entries of $\D$:
$$\|\D\|_{F_p} := \left(\sum_{ij} |\Delta_{ij}|^p\right)^{1/p}.$$
Analogous to before, $F_{p^*}$ is dual to $F_{p}$, where $1/p+1/p^*=1$.

\item The $p$-spectral (Schatten) norm, denoted $\|\cdot\|_{\sigma_p}$, is the $\ell_p$ norm on the singular values of the matrix $\D$:
$$\|\D\|_{\sigma_p}:= \|\bs\mu(\D)\|_p,$$
where $\bs\mu(\D)$ denotes the vector containing the singular values of $\D$. Again, $\sigma_{p^*}$ is dual to $\sigma_{p}$.

\item Finally we consider the class of induced norms. If $g:\R^{m}\to\R$ and $h:\R^n\to\R$ are norms, then we define the induced norm $\|\cdot\|_{\ind(h,g)}$ as
$$\|\D\|_{\ind(h,g)}:= \max_{\bb\in\R^n} \frac{g(\D\bb)}{h(\bb)}.$$
An important special case occurs when $g=\ell_p$ and $h=\ell_q$. When such norms are used, $\ind(q,p)$ is used as shorthand to denote $\ind(\ell_q,\ell_p)$. Induced norms are sometimes referred to as operator norms. We reserve the term operator norm for the induced norm $\ind(\ell_2,\ell_2) = (2,2) = \sigma_\infty$, which measures the largest singular value.
\end{enumerate}

\begin{table}[h]
\centering
\begin{tabular}{|c|c|c|c|}\hline
Name & Notation & Definition & Description\\\hline
$p$-Frobenius & $F_p$ & $\ds\left(\sum_{ij} |\Delta_{ij}|^p\right)^{1/p}$& entrywise $\ell_p$ norm\\\hline
$p$-spectral (Schatten) & $\sigma_p$ & $\|\bs\mu(\D)\|_p$ &$\ell_p$ norm on the singular values \\\hline
Induced & $(h,g)$ & $\ds\max_{\bb} \frac{g(\D\bb)}{h(\bb)}$ & induced by norms $g,h$\\\hline
\end{tabular}
\caption{Matrix norms on $\D\in\R^{m\times n}$.}
\label{tab:norms}
\end{table}

\subsection{Uncertain regression}

We now turn our attention to uncertain linear regression problems and regularization. The starting point for our discussion is the standard problem
$$\min_{\bb\in\R^n} g(\y-\X\bb),$$
where $\y\in\R^m$ and $\X\in\R^{m\times n}$ are data and $g$ is some convex function, typically a norm. For example, $g=\ell_2$ is least squares, while $g=\ell_1$ is known as least absolute deviation (LAD). In favor of models which mitigate the effects of overfitting these are often replaced by the \emph{regularization} problem
$$\min_{\bb} g(\y-\X\bb)+h(\bb),$$
where $h:\R^n\to\R$ is some penalty function, typically taken to be convex. This approach often aims to address overfitting by penalizing the complexity of the model, measured as $h(\bb)$. (For a more formal treatment using Hilbert space theory, see \cite{bousquet,bauschke}.) For example, taking $g=\ell_2^2$ and $h = \ell_2^2$, we recover the so-called regularized least squares (RLS), also known as ridge regression \cite{hastie}. The choice of $g=\ell_2^2$ and $h=\ell_1$ leads to Lasso, or least absolute shrinkage and selection operator, introduced in \cite{tibshirani}. Lasso is often employed in scenarios where the solution $\bb$ is desired to be sparse, i.e., $\bb$ has very few nonzero entries. Broadly speaking, regularization can take much more general forms; for our purposes, we restrict our attention to regularization that appears in the penalized form above.

In contrast to this approach, one may alternatively wish to re-examine the nominal regression problem $\min_{\bb}g(\y-\X\bb)$ and instead attempt to solve this taking into account adversarial noise in the data matrix $\X$. As in \cite{ghaoui,al1,al2,RObook,xu}, this approach may take the form
\begin{equation}\label{tenet}
\min_{\bb} \max_{\D\in\U} g(\y-(\X+\D)\bb),
\end{equation}
where the set $\U\sub\R^{m\times n}$ characterizes the user's belief about uncertainty on the data matrix $\X$. This set $\U$ is known  in the language of robust optimization \cite{RObook,ROsurvey} as an uncertainty set and the inner maximization problem $\max_{\D\in\U} g(\y-(\X+\D)\bb)$ takes into account the worst-case error (measured via $g$) over $\U$. We call such a procedure \emph{robustification} because it attempts to immunize or robustify the regression problem from structural uncertainty in the data. Such an adversarial or ``worst-case'' procedure is one of the key tenets of the area of robust optimization \cite{RObook,ROsurvey}.

As noted in the introduction, the adversarial perspective offers several attractive features. Let us first focus on settings when robustification coincides with a regularization problem. In such a case, the robustification identifies the adversarial perturbations the model is protected against, which can in turn provide additional insight into the behavior of different regularizers. Further, technical machinery developed for the construction of data-driven uncertainty sets in robust optimization \cite{ddunc1,ddunc2} enables the potential for a principled framework for the design of regularization schemes, in turn addressing a complex modeling decision encountered in practice.

Moreover, the adversarial approach is of interest in its own right, even if robustification does not correspond directly to a regularization problem. This is evidenced in part by the burgeoning success of generative adversarial networks and other methodologies in deep learning \cite{gan1,gan2,gan3}. Further, the worst-case approach often leads to a more straightforward analysis of properties of estimators \cite{xu} as well as algorithms for finding estimators \cite{bta}.

Let us now return to the robustification problem. A natural choice of an uncertainty set which gives rise to interpretability is the set $\U = \{\D\in\R^{m\times n}: \|\D\|\leq \lambda\}$, where $\|\cdot\|$ is some matrix norm and $\lambda>0$. One can then write $\max_{\D\in \U} g(\y-(\X+\D)\bb)$ as
$$\begin{array}{ll}
\ds\max_{\widetilde{\X}}&g(\y- \widetilde\X\bb)\\
\st&\|\X-\widetilde \X\|\leq \lambda,
\end{array}$$
or the worst case error taken over all $\widetilde \X$ sufficiently close to the data matrix $\X$. In what follows, if $\|\cdot\|$ is a norm or seminorm, then we let $\U_{\|\cdot\|}$ denote the ball of radius $\lambda$ in $\|\cdot\|$:
$$\U_{\|\cdot\|} = \{\D:\|\D\|\leq \lambda\}.$$
For example, $\U_{F_p}$, $\U_{\sigma_p}$, and $\U_{\ind(h,g)}$ denote uncertainty sets under the norms $F_p$, $\sigma_p$, and $\ind(h,g)$, respectively. We assume $\lambda>0$ fixed for the remainder of the paper. 

We briefly mention addressing uncertainty in $\y$. Suppose that we have a set $\mathcal{V}\sub\R^m$ which captures some belief about the uncertainty in $\y$. If again we have an uncertainty set $\U\sub\R^{m\times n}$, we may attempt to solve a problem of the form
$$\min_{\bb} \max_{\substack{\bs\delta\in\mathcal{V}\\\D\in\U}} g(\y+\bs\delta - (\X+\D)\bb).$$
We can instead work with a new loss function $\bar g$ defined as
$$\bar g(\vv) := \max_{\bs\delta\in\mathcal{V}} g(\vv + \bs\delta).$$
If $g$ is convex, then so is $\bar g$. In this way, we can work with the problem in the form
$$\min_{\bb} \max_{\D\in\U} \bar g(\y - (\X+\D)\bb),$$
where there is only uncertainty in $\X$. Throughout the remainder of this paper we will only consider such uncertainty.

\subsubsection*{Relation to robust statistics}

There has been extensive work in the robust statistics community on statistical methods which perform well in noisy, real-world environments. As noted in \cite{RObook}, the connection between robust optimization and robust statistics is not clear. We do not put forth any connection here, but briefly describe the development of robust statistics to appropriately contextualize our work. Instead of modeling noise via a distributional perspective, as is often the case in robust statistics, in this paper we choose to model it in a deterministic way using uncertainty sets. For a comprehensive description of the theoretical developments in robust statistics in the last half century, see the texts \cite{huber,rousseeuw} and the surveys \cite{robStatsSurvey,roussSurvey}.

A central aspect of work in robust statistics is the development and use of a more general set of loss functions. (This is in contrast to the robust optimization approach, which generally results in the same nominal loss function with a new penalty; see Section 2.3 below.) For example, while least squares (the $\ell_2$ loss) is known to perform well under Gaussian noise, it does not perform well under other types of noise, such as contaminated Gaussian noise. (Indeed, the Gaussian distribution was defined so that least squares is the optimal method under Gaussian noise \cite{rousseeuw}.) In contrast, a method like LAD regression (the $\ell_1$ loss) generally performs better than least squares with errors in $\y$, but not necessarily errors in the data matrix $\X$.

A more general class of such methods is $M$-estimators as proposed in \cite{huberMest} and since studied extensively 
\cite{huber,robRegBook}. However, $M$-estimators lack desirable \emph{finite sample breakdown} properties; in short, $M$-estimators perform very poorly in recovering the loadings $\bb^*$ under gross errors in the data $(\X,\y)$. To address some of these shortcomings, $GM$-estimators were introduced \cite{mallows,hill,hampel}. Since these, many other estimators have been proposed. One such method is least quantile of squares regression \cite{rousseeuw} which has highly desirable robustness properties. There has been significant interest in new robust statistical methods in recent years with the increasing availability of large quantities of high-dimensional data, which often make reliable outlier detection difficult. For commentary on modern approaches to robust statistics, see \cite{roussSurvey,bradicFanWang,fanFanBarut} and references therein.

\subsubsection*{Relation to error-in-variable models}

Another class of statistical models which are particularly relevant for the work contained herein are error-in-variable models \cite{eivb}. One approach to such a problem takes the form
$$\min_{\bb\in\R^n,\; \D\in\R^{m \times n}}g(\y - (\X+\D)\bb) + P(\D),$$
where $P$ is a penalty function which takes into account the complexity of possible perturbations $\D$ to the data matrix $\X$. A canonical example of such a method is total least squares \cite{tls,tlsoverview}, which can be written for fixed $\tau>0$ as
$$\min_{\bb,\D} \|\y-(\X+\D)\bb\|_2 + \tau\|\D\|_F.$$

An equivalent way of writing such problems is, instead of penalized form, as constrained optimization problems. In particular, the constrained version generically takes the form
\begin{equation}\label{eiv}
\min_{\bb} \min_{\substack{\D:\\ P(\D)\leq \eta}} g(\y-(\X+\D)\bb),
\end{equation}
where $\eta>0$ is fixed. Under the representation in \eqref{eiv}, the comparison with the robust optimization approach in \eqref{tenet} becomes immediate. While the classical error-in-variables approach takes an optimistic view on uncertainty in the data matrix $\X$, and finds loadings $\bb$ on the new ``corrected'' data matrix $\X+\D$, the minimax approach of \eqref{tenet} considers protections against adversarial perturbations in the data which maximally increase the loss.

One of the advantages of the adversarial approach to error-in-variables is that it enables a direct analysis of certain statistical properties, such as asymptotic consistency of estimators (\emph{c.f.} \cite{xu,xuChap}). In contrast, analyzing the consistency of estimators attained by a model such as total least squares is a complex issue \cite{tlsc}.

\subsection{Equivalence of robustification and regularization}

A natural question is when do the procedures of regularization and robustification coincide. This problem was first studied in \cite{ghaoui} in the context of uncertain least squares problems and has been extended to more general settings in \cite{xu,xuChap} and most comprehensively in \cite{RObook}. In this subsection, we present settings in which robustification is equivalent to regularization. When such an equivalence holds, tools from robust optimization can be used to analyze properties of the regularization problem (\emph{c.f.} \cite{xu,xuChap}).

We begin with a general result on robustification under induced seminorm uncertainty sets.

\begin{theorem}\label{thm:basicInducedNorm}
If $g:\R^m\to\R$ is a seminorm which is not identically zero and $h:\R^n\to\R$ is a norm, then for any $\zz\in\R^m$ and $\bb\in\R^n$
$$\max_{\D\in\U_{\ind(h,g)}} g(\zz+\D\bb) = g(\zz)+\lambda h(\bb),$$
where $\U_{\ind(h,g)} = \{\D:\|\D\|_{\ind(h,g)}\leq \lambda\}$.
\end{theorem}
\begin{proof}
From the triangle inequality $g(\zz+\D\bb)\leq g(\zz)+g(\D\bb)\leq g(\zz)+\lambda h(\bb)$ for any $\D\in\U:=\U_{\ind(h,g)}$. We next show that there exists some $\D\in\U$ so that $g(\zz+\D\bb) = g(\zz)+\lambda h(\bb)$. Let $\vv\in\R^n$ so that $\vv\in\operatorname{argmax}_{h^*(\vv)=1} \vv'\bb$, where $h^*$ is the dual norm of $h$. Note in particular that $\vv'\bb=h(\bb)$ by the definition of the dual norm $h^*$. For now suppose that $g(\zz)\neq0$. Define the rank one matrix $\hd = \frac{\lambda}{g(\zz)}\zz\vv'$. Observe that
$$g(\zz+\hd\bb) = g \left(\zz+\frac{\lambda h(\bb)}{g(\zz)}\zz\right) = \frac{g(\zz)+\lambda h(\bb)}{g(\zz)}g(\zz) =	g(\zz)+\lambda h(\bb).$$
We next show that $\hd\in\U$. Observe that for any $\xx\in\R^n$ that
$$g(\hd\xx) = g\left(\frac{\lambda \vv'\xx}{g(\zz)}\zz\right) = \lambda|\vv'\xx|\leq \lambda h(\xx)h^*(\vv) = \lambda h(\xx),$$
where the final inequality follows by definition of the dual norm. Hence $\hd\in\U$, as desired.

We now consider the case when $g(\zz)=0$. Let $\uu\in\R^m$ so that $g(\uu)=1$ (because $g$ is not identically zero there exists some $\uu$ so that $g(\uu)>0$, and so by homogeneity of $g$ we can take $\uu$ so that $g(\uu)=1$). Let $\vv$ be as before. Now define $\hd = \lambda\uu\vv'$. We observe that
$$g(\zz+\hd\bb) = g(\zz+\lambda \uu\vv'\bb)\leq g(\zz)+\lambda |\vv'\bb|g(\uu) = \lambda h(\bb).$$
Now, by the reverse triangle inequality,
$$g(\zz+\hd\bb) \geq g(\hd\bb) - g(\zz) = g(\hd\bb) = \lambda h(\bb),$$
and therefore $g(\zz+\hd\bb)=\lambda h(\bb) = g(\zz)+\lambda h(\bb)$. The proof that $\hd\in\U$ is identical to the case when $g(\zz)\neq0$. This completes the proof.
\end{proof}

This result implies as a corollary known results on the connection between robustification and regularization as found in \cite{xu,RObook,xuChap} and references therein.
\begin{corollary}[\cite{xu,RObook,xuChap}]\label{corr:basicInducedNorm}
If $p,q\in[1,\infty]$ then
$$\min_{\bb} \max_{\D\in\U_{\ind(q,p)}} \|\y-(\X+\D)\bb\|_p = \min_\bb \|\y-\X\bb\|_p+\lambda \|\bb\|_q.$$
In particular, for $p=q=2$ we recover regularized least squares as a robustification; likewise, for $p=2$ and $q=1$ we recover the Lasso.\footnote{Strictly speaking, we recover \emph{equivalent} problems to regularized least squares and Lasso, respectively. We take the usual convention and overlook this technicality (see \cite{RObook} for a discussion). For completeness, we note that one can work directly with the true $\ell_2^2$ loss function, although at the cost of requiring more complicated uncertainty sets to recover equivalence results.}
\end{corollary}

\begin{theorem}[\cite{xu,RObook,xuChap}]\label{thm:basicLpResults}
One has the following for any $p,q\in[1,\infty]$:
$$\min_\bb \max_{\D\in\U_{F_p}} \|\y-(\X+\D)\bb\|_p = \min_\bb \|\y-\X\bb\|_p+ \lambda \|\bb\|_{p^*},$$
where $p^*$ is the conjugate of $p$. Similarly,
$$\min_\bb \max_{\D\in \U_{\sigma_q}}\|\y-(\X+\D)\bb\|_2 = \min_\bb \|\y-\X\bb\|_2 + \lambda \|\bb\|_2.$$
\end{theorem}

Observe that regularized least squares arises again under all uncertainty sets defined by the spectral norms $\sigma_q$ when the loss function is $g=\ell_2$. Now we continue with a remark on how Lasso arises through regularization. See \cite{xu} for comprehensive work on the robustness and sparsity implications of Lasso as interpreted through such a robustification considered in this paper.

\begin{remark}\label{remark:ell1ComesFromEll0}
As per Corollary \ref{corr:basicInducedNorm} it is known that Lasso arises as uncertain $\ell_2$ regression with uncertainty set $\U:=\U_{\ind(1,2)}$ \cite{xu}. As with Theorem \ref{thm:basicInducedNorm}, one might argue that the $\ell_1$ penalizer arises as an artifact of the model of uncertainty. We remark that one can derive the set $\U$ as an induced uncertainty set defined using the ``true'' non-convex penalty $\ell_0$, where $\|\bb\|_0:=|\{i:\beta_i\neq0\}|$. To be precise, for any $p\in[1,\infty]$ and for $\Gamma = \{\bb\in\R^n: \|\bb\|_p\leq 1\}$ we claim that
$$\U' := \left\{ \D: \max_{\bb\in\Gamma} \frac{\|\D\bb\|_2}{\|\bb\|_0}\leq \lambda\right\}$$
satisfies $\U=\U'$. This is summarized, with an additional representation $\U''$ as used in \cite{xu}, in the following proposition.
\begin{proposition}
If $\U=\U_{(1,2)}$, $\U' = \{\D: \|\D\bb\|_2\leq \lambda \|\bb\|_0\;\forall \|\bb\|_p\leq 1\}$ for an arbitrary $p\in[1,\infty]$, and $\U'' = \{\D:\|\D_i\|_2\leq \lambda \;\forall i\}$, where $\D_i$ is the $i$th column of $\D$, then $\U=\U'=\U''$.
\end{proposition}
\begin{proof}
We first show that $\U=\U'$. Because $\|\bb\|_1\leq \|\bb\|_0$ for all $\bb\in\R^n$ with $\|\bb\|_p\leq 1$, we have that $\U\sub\U'$. Now suppose that $\D \in \U'$. Then for any $\bb\in\R^n$, we have that
$$\|\bDelta \bbeta \|_2= \left\|\sum_i \beta_i\bDelta \mb e_i\right\|_2 \leq \sum_i|\beta_i|\left\| \bDelta \mb e_i\right\|_2\leq\sum_i|\beta_i|\lambda = \lambda \|\bbeta\|_1,$$
where $\{\mb e_i\}_{i=1}^n$ is the standard orthonormal basis for $\R^n$. Hence, $\D\in\U$ and therefore $\U'\sub\U$. Combining with the previous direction gives $\U=\U'$.

We now prove that $\U=\U''$. That $\U''\sub \U$ is essentially obvious; $\U\sub \U''$ follows by considering $\bb\in\{\mb e_i\}_{i=1}^n$. This completes the proof.
\end{proof}

\noindent This proposition implies that $\ell_1$ arises from the robustification setting without directly appealing to standard convexity arguments for why $\ell_1$ should be used to replace $\ell_0$ (which use the fact that $\ell_1$ is the so-called convex envelope of $\ell_0$ on $[-1,1]^n$, see e.g. \cite{boydvandenberghe}).
\end{remark}

In light of the above discussion, it is not difficult to show that other Lasso-like methods can also be expressed as an adversarial robustification, supporting the flexibility and versatility of such an approach. One such example is the elastic net \cite{zhel,demol,mosci}, a hybridized version of ridge regression and the Lasso. An equivalent representation of the elastic net is as follows:
$$\min_\bb \|\y-\X\bb\|_2 + \lambda\|\bb\|_1+\mu\|\bb\|_2.$$
As per Theorem \ref{thm:basicLpResults}, this can be written exactly as
$$\min_\bb \max_{\substack{\D,\D':\\\|\D\|_{F_\infty} \leq \lambda\\ \|\D'\|_{F_2}\leq \mu }}\|\y-(\X+\D+\D')\bb\|_2.$$
Under this interpretation, we see that $\lambda$ and $\mu$ directly control the tradeoff between two different types of perturbations: ``feature-wise'' perturbations $\D$ (controlled via $\lambda$ and the $F_\infty$ norm) and ``global'' perturbations $\D'$ (controlled via $\mu$ and the $F_2$ norm).

We conclude this subsection with another example of when robustification is equivalent to regularization for the case of LAD ($\ell_1$) and maximum absolute deviation ($\ell_\infty$) regression under row-wise uncertainty.

\begin{theorem}[\cite{xu}]\label{thm:rowWiseUncertainty}
Fix $q\in[1,\infty]$ and let $\U = \{\D: \|\bd_i\|_q\leq \lambda\;\forall i\}$, where $\bd_i$ is the $i$th row of $\D\in\R^{m\times n}$. Then
$$\min_\bb\max_{\D\in\U} \|\y-(\X+\D)\bb\|_1  = \min_\bb \|\y-\X\bb\|_1 + m\lambda\|\bb\|_{q^*}$$
and
$$\min_\bb\max_{\D\in\U} \|\y-(\X+\D)\bb\|_\infty  = \min_\bb \|\y-\X\bb\|_\infty + \lambda\|\bb\|_{q^*}.$$
\end{theorem}

\noindent For completeness, we note that the uncertainty set $\U= \{\D:\|\bd_i\|_q\leq \lambda\;\forall i\}$ considered in Theorem \ref{thm:rowWiseUncertainty} is actually an induced uncertainty set, namely, $\U = \U_{(q^*,\infty)}.$

\subsection{Non-equivalence of robustification and regularization}

In contrast to previous work studying robustification for regression, which primarily addresses tractability of solving the new uncertain problem \cite{RObook} or the implications for Lasso \cite{xu}, we instead focus our attention on characterization of the equivalence between robustification and regularization. We begin with a regularization upper bound on robustification problems.

\begin{proposition}\label{prop:alwaysUpperBound}
Let $\U\sub\R^{m\times n}$ be any non-empty, compact set and $g:\R^m\to\R$ a seminorm. Then there exists some seminorm $\oh:\R^n\to\R$ so that for any $\zz\in\R^m$, $\bb\in\R^n$,
$$\max_{\D\in\U} g(\zz+\D\bb) \leq g(\zz) + \oh(\bb),$$
with equality when $\zz=\mb0$.
\end{proposition}

\begin{proof}
Let $\oh:\R^n\to\R$ be defined as
$$\oh(\bb) := \max_{\D\in\U} g(\D\bb).$$
To show that $\oh$ is a seminorm we must show it satisfies absolute homogeneity and the triangle inequality. For any $\bb\in\R^n$ and $\alpha\in\R$,
$$\oh(\alpha \bb) = \max_{\D\in\U} g(\D(\alpha\bb)) = \max_{\D\in\U} |\alpha| g(\D\bb) =|\alpha|\left(\max_{\D\in\U} g(\D\bb)\right) = |\alpha|\oh(\bb),$$
so absolute homogeneity is satisfied. Similarly, if $\bb,\gam\in\R^n$,
\begin{align*}
\oh(\bb+\gam) =\max_{\D\in\U} g(\D(\bb+\gam)) &\leq  \max_{\D\in\U} \left[g(\D\bb)+g(\D\gam)\right]\\
&\leq \left(\max_{\D\in\U} g(\D\bb)\right) + \left(\max_{\D\in\U} g(\D\gam)\right),
\end{align*}
and hence the triangle inequality is satisfied. Therefore, $\oh$ is a seminorm which satisfies the desired properties, completing the proof.
\end{proof}

When equality is attained for all pairs $(\zz,\bb)\in\R^m\times\R^n$, we are in the regime of the previous subsection, and we say that robustification under $\U$ is equivalent to regularization under $\oh$. We now discuss a variety of explicit settings in which regularization only provides upper and lower bounds to the true robustified problem.

Fix $p,q\in[1,\infty]$. Consider the robust $\ell_p$ regression problem
$$\min_\bb \max_{\D\in \U_{F_q}} \|\y-(\X+\D)\bb\|_p,$$
where $\U_{F_q} = \{\D\in\R^{m\times n}: \|\D\|_{F_q}\leq \lambda\}$. In the case when $p=q$ we saw earlier (Theorem \ref{thm:basicLpResults}) that one exactly recovers $\ell_p$ regression with an $\ell_{p^*}$ penalty:
$$\min_\bb \max_{\D\in\U_{F_p}} \|\y-(\X+\D)\bb\|_p = \min_\bb \|\y-\X\bb\|_p+\lambda\|\bb\|_{p^*}.$$
Let us now consider the case when $p\neq q$. We claim that regularization (with $\oh$) is no longer equivalent to robustification (with $\U_{F_q}$) unless $p\in\{1,\infty\}$. Applying Proposition \ref{prop:alwaysUpperBound}, one has for any $\zz\in\R^m$ that
$$\max_{\D\in\U_{F_q}} \|\zz+\D\bb\|_p \leq \|\zz\|_p + \oh(\bb),$$
where $\oh=\max_{\D\in \U_{F_q}} \|\D\bb\|_p$ is a norm (when $p=q$, this is precisely the $\ell_{p^*}$ norm, multiplied by $\lambda$). Here we can compute $\oh$. To do this we first define a discrepancy function as follows:

\begin{definition}
For $a,b\in[1,\infty]$ define the discrepancy function $\delta_m(a,b)$ as
$$\delta_m(a,b):=\max\{\|\uu\|_a:\uu\in\R^m,\;\|\uu\|_b = 1\}.$$
\end{definition}

This discrepancy function is computable and well-known (see e.g. \cite{hornjohnson}):
$$\delta_m(a,b) = \left\{\begin{array}{ll}
m^{ 1/a-1/b},& \text{ if }a \leq b\\
1,& \text{ if }a > b.
\end{array}\right.$$
It satisfies $1\leq \delta_m(a,b)\leq m$ and $\delta_m(a,b)$ is continuous in $a$ and $b$. One has that $\delta_m(a,b)=\delta_m(b,a)=1$ if and only if $a=b$ (so long as $m\geq2$). Using this, we now proceed with the theorem. The proof applies basic tools from real analysis and is contained in Appendix A.

\begin{theorem}\label{thm:discrepLbUb}

\begin{enumerate}[(a)]
\item For any $\zz\in\R^m$ and $\bb\in\R^n$,
\begin{equation}\label{eqn:discUB}
\max_{\D\in\U_{F_q}} \|\zz+\D\bb\|_p \leq \|\zz\|_p + \lambda\delta_m(p,q)\|\bb\|_{q^*}.
\end{equation}

\item When $p\in\{1,\infty\}$, there is equality in \eqref{eqn:discUB} for all $(\zz,\bb)$.

\item When $p\in(1,\infty)$ and $p\neq q$, for any $\bb\neq\mb0$ the set of $\zz\in\R^m$ for which the inequality \eqref{eqn:discUB} holds at equality is a finite union of one-dimensional subspaces (so long as $m\geq2$). Hence, for any $\bb\neq\mb0$ the inequality in \eqref{eqn:discUB} is strict for almost all $\zz$.

\item For $p\in(1,\infty)$, one has for all $\zz\in\R^m$ and $\bb\in\R^n$ that
\begin{equation}\label{eqn:discLB}
\|\zz\|_p + \frac{\lambda}{\delta_m(q,p)}\|\bb\|_{q^*}\leq\max_{\D\in\U_{F_q}} \|\zz+\D\bb\|_p.
\end{equation}

\item For $p\in(1,\infty)$, the lower bound in \eqref{eqn:discLB} is best possible in the sense that the gap can be arbitrarily small, i.e., for any $\bb\in\R^n$,
$$\inf_{\zz}\left(\max_{\D\in\U_{F_q}} \|\zz+\D\bb\|_p-\|\zz\|_p - \frac{\lambda}{\delta_m(q,p)}\|\bb\|_{q^*}\right)=0.$$

\end{enumerate}

\end{theorem}

Theorem \ref{thm:discrepLbUb} characterizes precisely when robustification under $\U_{F_q}$ is equivalent to regularization for the case of $\ell_p$ regression. In particular, when $p\neq q$ and $p\in(1,\infty)$, the two are \emph{not} equivalent, and one only has that
\begin{align*}
\min_{\bb} \|\y-\X\bb\|_p + \frac{\lambda}{\delta_m(q,p)}\|\bb\|_{q^*}&\leq \min_\bb\max_{\D\in\U_{F_q}} \|\y-(\X+\D)\bb\|_p \\
&\leq \min_\bb \|\y-\X\bb\|_p + \lambda\delta_m(p,q)\|\bb\|_{q^*}.
\end{align*}
Further, we have shown that these upper and lower bounds are the \emph{best possible} (Theorem \ref{thm:discrepLbUb}, parts (c) and (e)). While $\ell_p$ regression with uncertainty set $\U_{F_q}$ for $p\neq q$ and $p\in(1,\infty)$ still has both upper and lower bounds which correspond to regularization (with different regularization parameters $\ol\in\left[{\lambda}/{\delta_m(q,p)},\lambda\delta_m(p,q)\right]$), we emphasize that in this case there is no longer the direct connection between the parameter garnering the magnitude of uncertainty ($\lambda$) and the parameter for regularization ($\ol$).

\begin{example}
As a concrete example, consider the implications of Theorem \ref{thm:discrepLbUb} when $p=2$ and $q=\infty$. We have that
\begin{align*}
\min_{\bb} \|\y-\X\bb\|_2 + \lambda \|\bb\|_1 &\leq \min_{\bb} \max_{\D\in\U_{F_\infty}} \|\y-(\X+\D)\bb\|_p\\
& \leq \min_{\bb} \|\y-\X\bb\|_2 + \sqrt{m}\lambda \|\bb\|_1.
\end{align*}
In this case, robustification is \emph{not} equivalent to regularization. In particular, in the regime where there are many data points (i.e. $m$ is large), the gap appearing between the different problems can be quite large.
\end{example}

Let us remark that in general, lower bounds on $\max_{\D\in\U} g(\zz+\D\bb)$ will depend on the structure of $\U$ and may not exist (except for the trivial lower bound of $g(\zz)$) in some scenarios. However, it is easy to show that if $\U$ is compact and zero is in the interior of $\U$, then there exists some $\ul\in(0,1]$ so that
$$\max_{\D\in\U} g(\zz+\D\bb)\geq g(\zz) + \ul \oh (\bb).$$

Before proceeding with other choices of uncertainty sets, it is important to make a further distinction about the general non-equivalence of robustification and regularization as presented in 
Theorem \ref{thm:discrepLbUb}. In particular, it is simple to construct examples (see Appendix B) which imply the following strong existential result:

\begin{theorem}\label{thm:regPaths}
In a setting when robustification and regularization are not equivalent, it is possible for the two problems to have no common optimal solution. In particular,
$$\bb^*\in\underset{\bb}{\operatorname{argmin} } \ds\max_{\D\in\U} g(\y-(\X+\D)\bb)$$
is not necessarily a solution of
$$\min_{\bb} g(\y-\X\bb) + \widetilde{\lambda}\oh(\bb)$$
for \emph{any} $\widetilde{\lambda}>0$, and vice versa. 
\end{theorem}

In other words, the regularization path (as $\widetilde\lambda\in(0,\infty)$ varies) and the robustification path (as the radius $\lambda\in(0,\infty)$ of $\U$ varies) do not necessarily have \emph{any} point in common. As a result, when robustification and regularization do not coincide, they can induce structurally distinct solutions.

We now proceed to analyze another setting in which robustification is not equivalent to regularization. The setting, in line with Theorem \ref{thm:basicLpResults}, is $\ell_p$ regression under spectral uncertainty sets $\U_{\sigma_q}$. As per Theorem \ref{thm:basicLpResults}, one has that
$$\min_\bb\max_{\D\in\U_{\sigma_q}} \|\y-(\X+\D)\bb\|_2 = \min_\bb \|\y-\X\bb\|_2 + \lambda \|\bb\|_2$$
for any $q\in[1,\infty]$. This result on the ``universality'' of RLS under a variety of uncertainty sets relies on the fact that the $\ell_2$ norm underlies spectral decompositions; namely, one can write any matrix $\X$ as $\sum_i\mu_i\uu_i\vv_i'$, where $\{\mu_i\}_i$ are the singular values of $\X$, $\{\uu_i\}_i$ and $\{\vv_i\}_i$ are the left and right singular vectors of $\X$, respectively, and $\|\uu_i\|_2 = \|\vv_i\|_2=1$ for all $i$.

A natural question is what happens when the loss function $\ell_2$, a modeling choice, is replaced by $\ell_p$, where $p\in[1,\infty]$.  We claim that for $p\notin\{1,2,\infty\}$, robustification under $\U_{\sigma_q}$ is no longer equivalent to regularization. In light of Theorem \ref{thm:discrepLbUb} this is not difficult to prove. We find that the choice of $q\in[1,\infty]$, as before, is inconsequential. We summarize this in the following proposition:

\begin{proposition}\label{prop:nonequivRLSuniversal}
For any $\zz\in\R^m$ and $\bb\in\R^n$,
\begin{equation}\label{eqn:discUB2}
\max_{\D\in\U_{\sigma_q}} \|\zz+\D\bb\|_p \leq \|\zz\|_p + \lambda\delta_m(p,2)\|\bb\|_{2}.
\end{equation}
In particular, if $p\in\{1,2,\infty\}$, there is equality in \eqref{eqn:discUB2} for all $(\zz,\bb)$. If $p\notin\{1,2,\infty\}$, then for any $\bb\neq\mb0$ the inequality in \eqref{eqn:discUB2} is strict for almost all $\zz$ (when $m\geq 2$). Further, for $p\notin\{1,2,\infty\}$ one has the lower bound
$$\|\zz\|_p + \frac{\lambda}{\delta_m(2,p)}\|\bb\|_{2}\leq\max_{\D\in\U_{\sigma_q}} \|\zz+\D\bb\|_p,$$
whose gap is arbitrarily small for all $\bb$.
\end{proposition}

\begin{proof}
This result is Theorem \ref{thm:discrepLbUb} in disguise. This follows by noting that
$$\max_{\D\in\U_{\sigma_q}} \|\zz+\D\bb\|_p = \max_{\D\in\U_{F_2}}\|\zz+\D\bb\|_p$$
and directly applying the preceding results.
\end{proof}

We now consider a third setting for $\ell_p$ regression, this time subject to uncertainty $\U_{\ind(q,r)}$; this is a generalized version of the problems considered in Theorems \ref{thm:basicInducedNorm} and \ref{thm:rowWiseUncertainty}. From Theorem \ref{thm:basicInducedNorm} we know that if $p=r$, then
$$\min_\bb \max_{\D\in\U_{\ind(q,p)}} \|\y-(\X+\D)\bb\|_p = \min_\bb \|\y-\X\bb\|_p + \lambda \|\bb\|_q.$$
Similarly, as per Theorem \ref{thm:rowWiseUncertainty}, when $r=\infty$ and $p\in\{1,\infty\}$,
$$\min_\bb \max_{\D\in\U_{(q,\infty)}} \|\y-(\X+\D)\bb\|_p = \min_\bb \|\y-\X\bb\|_p + \lambda\delta_m(p,\infty)\|\bb\|_q.$$
Given these results, it is natural to inquire what happens for more general choices of induced uncertainty set $\U_{\ind(q,r)}$. As before with Theorem \ref{thm:discrepLbUb}, we have a complete characterization of the equivalence of robustification and regularization for $\ell_p$ regression with uncertainty set $\U_{\ind(q,r)}$:
\begin{proposition}\label{prop:rowWiseUncFullCharac}
For any $\zz\in\R^m$ and $\bb\in\R^n$,
\begin{equation}\label{eqn:discUB3}
\max_{\D\in\U_{\ind(q,r)}} \|\zz+\D\bb\|_p \leq \|\zz\|_p + \lambda\delta_m(p,r)\|\bb\|_{q}.
\end{equation}
In particular, if $p\in\{1,r,\infty\}$, there is equality in \eqref{eqn:discUB2} for all $(\zz,\bb)$. If $p\in(1,\infty)$ and $p\neq r$, then for any $\bb\neq\mb0$ the inequality in \eqref{eqn:discUB3} is strict for almost all $\zz$ (when $m\geq 2$). Further, for $p\in(1,\infty)$ with $p\neq r$ one has the lower bound
$$\|\zz\|_p + \frac{\lambda}{\delta_m(r,p)}\|\bb\|_{q}\leq\max_{\D\in\U_{\ind(q,r)}} \|\zz+\D\bb\|_p,$$
whose gap is arbitrarily small for all $\bb$.
\end{proposition}

\begin{proof}
The proof follows the argument given in the proof of Theorem \ref{thm:discrepLbUb}. Here we simply note that now one uses the fact that
$$\max_{\D\in\U_{\ind(q,r)}} \|\zz+\D\bb\|_p = \max_{\|\uu\|_r \leq \lambda \|\bb\|_q } \|\zz+\uu\|_p.$$
\end{proof}

\noindent We summarize all of the results on linear regression in Table \ref{tab:linRegEquiv}.

\begin{table}[h]
\centering
\begin{tabular}{|c|c|c|c|}\hline
Loss function & Uncertainty set $\U$ & $\oh(\bb)$ & Equivalence if and only if\\\hline\hline
seminorm $g$& $\U_{\ind(h,g)}$ ($h$ norm) & $\lambda h(\bb)$ & always\\\hline
$\ell_p$&$\U_{\sigma_q}$&$\lambda\delta_m(p,2)\|\bb\|_2$ & $p\in\{1,2,\infty\}$\\\hline
$\ell_p$&$\U_{F_q}$&$\lambda \delta_m(p,q)\|\bb\|_{q^*}$& $p\in\{1,q,\infty\}$\\\hline
$\ell_p$&$\U_{\ind(q,r)}$&$\lambda \delta_m(p,r)\|\bb\|_q$& $p\in\{1,r,\infty\}$\\\hline
$\ell_p$&$\{\D:\|\bd_i\|_q\leq \lambda\;\forall i\}$& $\lambda m^{1/p}\|\bb\|_{q^*}$& $p\in\{1,\infty\}$\\\hline
\end{tabular}
\caption{Summary of equivalencies for robustification with uncertainty set $\U$ and regularization with penalty $\oh$, where $\oh$ is as given in Proposition \ref{prop:alwaysUpperBound}. Here by equivalence we mean that for all $\zz\in\R^m$ and $\bb\in\R^n$, $\max_{\D\in\U} g(\zz+\bb) = g(\zz)+\oh(\bb)$, where $g$ is the loss function, i.e., the upper bound $\oh$ is also a lower bound. Here $\delta_m$ is as in Theorem \ref{thm:discrepLbUb}. Throughout $p,q\in[1,\infty]$ and $m\geq2$. Here $\bd_i$ denotes the $i$th row of $\D$.}
\label{tab:linRegEquiv}
\end{table}


\section{On the equivalence of robustification and regularization in matrix estimation problems}

A substantial body of problems at the core of modern developments in statistical estimation involves underlying matrix variables. Two prominent examples which we consider here are matrix completion and Principal Component Analysis (PCA). In both cases we show that a common choice of the regularization problem corresponds exactly to a robustification of the nominal problem subject to uncertainty. In doing so we expand the existing knowledge of robustification for vector regression to a novel and substantial domain. We begin by reviewing these two problem classes before introducing a simple model of uncertainty analogous to the vector model of uncertainty.

\subsection{Problem classes}

In matrix completion problems one is given data $Y_{ij}\in\R$ for $(i,j)\in E\sub \{1,\ldots,m\}\times\{1,\ldots,n\}$. One problem of interest is rank-constrained matrix completion
\begin{equation}\label{eqn:RCMC}
\begin{array}{ll}
\ds\min_\X&\|\Y-\X\|_{\p(F_2)}\\
\st&\rank(\X)\leq k,
\end{array}
\end{equation}
where $\|\cdot\|_{\p(F_2)}$ denotes the projected $2-$Frobenius seminorm, namely,
$$\|\z\|_{\p(F_2)} = \left(\sum_{(i,j)\in E} Z_{ij}^2\right)^{1/2}.$$

Matrix completion problems appear in a wide variety of areas. One well-known application is in the Netflix challenge \cite{netflix}, where one wishes to predict user movie preferences based on a very limited subset of given user ratings. Here rank-constrained models are important in order to obtain parsimonious descriptions of user preferences in terms of a limited number of significant latent factors. The rank-constrained problem \eqref{eqn:RCMC} is typically converted to a regularized form with rank replaced by the nuclear norm $\sigma_1$ (the sum of singular values) to obtain the convex problem
$$\min_{\X} \|\Y-\X\|_{\p(F_2)}+\lambda \|\X\|_{\sigma_1}.$$
In what follows we show that this regularized problem can be written as an uncertain version of a nominal problem $\min_\X \|\Y-\X\|_{\p(F_2)}$.

Similarly to matrix completion, PCA typically takes the form
\begin{equation}\label{eqn:PCA}
\begin{array}{ll}
\ds\min_\X&\|\Y-\X\|\\
\st&\rank(\X)\leq k,
\end{array}
\end{equation}
where $\|\cdot\|$ is either the usual Frobenius norm $F_2=\sigma_2$ or the operator norm $\sigma_\infty$, and $\Y\in\R^{m\times n}$. PCA arises naturally by assuming that $\Y$ is observed as some low-rank matrix $\X$ plus noise: $\Y=\X+\mb E$. The solution to \eqref{eqn:PCA} is well-known to be a truncated singular value decomposition which retains the $k$ largest singular values \cite{eckart}. PCA is popular for a variety of applications where dimension reduction is desired.

A variant of PCA known as robust PCA \cite{candesrpca} operates under the assumption that some entries of $\Y$ may be grossly corrupted. Robust PCA assumes that $\Y=\X+\E$, where $\X$ is low rank and $\E$ is sparse (few nonzero entries). Under this model robust PCA takes the form
\begin{equation}\label{eqn:RobPCA}
\min_{\X} \|\Y-\X\|_{F_1}+ \lambda \|\X\|_{\sigma_1}.
\end{equation}
Here again we can interpret $\|\X\|_{\sigma_1}$ as a surrogate penalty for rank. In the spirit of results from compressed sensing on exact $\ell_1$ recovery, it is shown in \cite{candesrpca} that \eqref{eqn:RobPCA} can exactly recover the true $\X_0$ and $\E_0$ assuming that the rank of $\X_0$ is small, $\E_0$ is sufficiently sparse, and the eigenvectors of $\X_0$ are well-behaved (see technical conditions contained therein). Below we derive explicit expressions for PCA subject to certain types of uncertainty; in doing so we show that robust PCA does not correspond to an adversarially robust version of $\min_{\X} \|\Y-\X\|_{\sigma_\infty}$ or $\min_{\X} \|\Y-\X\|_{F_2}$ for \emph{any} model of additive linear uncertainty. 

Finally let us note that the results we consider here on robust PCA are distinct from considerations in the robust statistics community on robust approaches to PCA. For results and commentary on such methods, see \cite{croux,hubert,salibian,roussSurvey}.

\subsection{Models of uncertainty}

For these two problem classes we now detail a model of uncertainty. Our underlying problem is of the form $\min_\X \|\Y-\X\|$, where $\Y$ is given data (possibly with some unknown entries). As with the vector case, we do not concern ourselves with uncertainty in the observed $\Y$ because modeling uncertainty in $\Y$ simply leads to a different choice of loss function. To be precise, if $\mathcal{V}\sub\R^{m\times n}$ and $g$ is convex loss function then
$$\bar g(\Y-\X):=\max_{\D\in\mathcal{V}} g((\Y+\D) - \X)$$
is a new convex loss function $\bar g$ of $\Y-\X$.

As in the vector case we assume a linear model of uncertainty in the measurement of $\X$:
$$Y_{ij} = X_{ij} + \left(\sum_{\ell k} \Delta_{\ell k}^{(ij)} X_{\ell k}\right) + \epsilon_{ij},$$
where $\D^{(ij)}\in\R^{m\times n}$; alternatively, in inner product notation, $Y_{ij} = X_{ij} + \langle \D^{(ij)},\X\rangle + \epsilon_{ij}$. This linear model is in direct analogy with the model for vector regression taken earlier; now $\bb$ is replaced by $\X$, and again we consider linear perturbations of the unknown regression variable.

This linear model of uncertainty captures a variety of possible forms of uncertainty and accounts for possible interactions among different entries of the matrix $\X$. Note that in matrix notation, the nominal problem becomes, subject to linear uncertainty in $\X$,
$$\min_\X \max_{\D\in\U} \|\Y-\X-\D(\X)\|,$$
where here $\U$ is some collection of linear maps and $\D\in\U$ is defined as $[\D(\X)]_{ij} = \langle \D^{(ij)},\X\rangle $, where again $\D^{(ij)}\in\R^{m\times n}$ (all linear maps can be written in such a form). Note here the direct analogy to the vector case, with the notation $\D(\X)$ chosen for simplicity. (For clarity, note that $\D$ is not itself a matrix, although one could interpret it as a matrix in $\D^{mn\times mn}$, albeit at a notational cost; we avoid this here.)

We now outline some particular choices for uncertainty sets. As with the vector case, one natural set is an induced uncertainty set. Precisely, if $g,h:\R^{m\times n}\to\R$ are functions, then we define an induced uncertainty set
$$\U_{\ind(h,g)}:= \left\{\D:\R^{m\times n}\to\R^{m\times n}\;|\;\D\text{ linear, } g(\D(\X))\leq \lambda h(\X)\;\forall \X\in \R^{m\times n}\right\}.$$
As before, when $g$ and $h$ are both norms, $\U_{\ind(h,g)}$ is precisely a ball of radius $\lambda$ in the induced norm
$$\|\D\|_{\ind(h,g)} = \max_{\X} \frac{g(\D(\X))}{h(\X)}.$$
There are also many other possible choices of uncertainty sets. These include the spectral uncertainty sets
$$\U_{\sigma_p} = \{\D:\R^{m\times n}\to\R^{m\times n}|\D\text{ linear, }\|\D\|_{\sigma_p}\leq \lambda \},$$
where we interpret $\|\D\|_{\sigma_p}$ as the $\sigma_p$ norm of $\D$ in any, and hence all, of its matrix representations. Other uncertainty sets are those such as $\U = \{\D: \D^{(ij)}\in \U^{(ij)}\}$, where $\U^{(ij)}\sub\R^{m\times n}$ are themselves uncertainty sets. These last two models we will not examine in depth here because they are often subsumed by the vector results (note that these two uncertainty sets do not truly involve the matrix structure of $\X$, and can therefore be ``vectorized'', reducing directly to vector results).

\subsection{Basic results on equivalence}

We now continue with some underlying theorems for our models of uncertainty. As a first step, we provide a proposition on the spectral uncertainty sets. As noted above, this result is exactly Theorem \ref{thm:basicLpResults}, and therefore we will not consider such uncertainty sets for the remainder of the paper.
\begin{proposition}
For any $q\in[1,\infty]$ and any $\Y\in\R^{m\times n}$,
$$\min_{\X} \max_{\D\in\U_{\sigma_q}} \|\Y-\X-\D(\X)\|_{F_2} = \min_{\X} \|\Y-\X\|_{F_2} + \lambda \|\X\|_{F_2}.$$
\end{proposition}
For what follows, we restrict our attention to induced uncertainty sets. We begin with an analogous result to Theorem \ref{thm:basicInducedNorm}. The proof is similar and therefore kept concise. Throughout we always assume without loss of generality that if $Y_{ij}$ is not known then $Y_{ij}=0$ (i.e., we set it to some arbitrary value).
\begin{theorem}\label{thm:genMatrix}
If $g:\R^{m\times n} \to \R$ is a seminorm which is not indentically zero and $h:\R^{m\times n}\to\R$ is a norm, then
$$\min_{\X} \max_{\D\in\U_{\ind(h,g)}} g\left(\Y-\X-\D(\X)\right) = \min_{\X} g\left(\Y-\X\right) + \lambda h\left( \X\right).$$
\end{theorem}

\noindent This theorem leads to an immediate corollary:
\begin{corollary}\label{cor:matrix}
For any norm $\|\cdot\|:\R^{m\times n}\to\R$ and any $p\in[1,\infty]$
$$\min_{\X} \max_{\D\in\U_{\ind(\sigma_p,\|\cdot\|)}} \|\Y-\X-\D(\X)\| = \min_{\X} \|\Y-\X\| +\lambda\|\X\|_{\sigma_p}.$$
\end{corollary}
In the two subsections which follow we study the implications of Theorem \ref{thm:genMatrix} for matrix completion and PCA.

\subsection{Robust matrix completion}

We now proceed to apply Theorem \ref{thm:genMatrix} for the case of matrix completion. Note that the projected Frobenius ``norm'' $\p(F_2)$ is a seminorm. Therefore, we arrive at the following corollary:

\begin{corollary}\label{cor:robMC}
For any $p\in[1,\infty]$ one has that
$$\min_{\X} \max_{\D\in\U_{\ind(\sigma_p,\p(F_2))}} \|\Y-\X-\D(\X)\|_{\p(F_2)} = \min_{\X} \|\Y-\X\|_{\p(F_2)} + \lambda \|\X\|_{\sigma_p}.$$
In particular, for $p=1$ one exactly recovers so-called nuclear norm penalized matrix completion:
$$\min_{\X} \|\Y-\X\|_{\p(F_2)}+\lambda \|\X\|_{\sigma_1}.$$
\end{corollary}

It is not difficult to show by modifying the proof of Theorem \ref{thm:genMatrix} that even though $\U_{\ind(\sigma_p,F_2)} \subsetneq \U_{\ind(\sigma_p,\p(F_2) )}$, the following holds:
\begin{proposition}\label{prop:robMC}
For any $p\in[1,\infty]$ one has that
$$\min_{\X} \max_{\D\in\U_{\ind(\sigma_p,F_2)}} \|\Y-\X-\D(\X)\|_{\p(F_2)} = \min_{\X} \|\Y-\X\|_{\p(F_2)} + \lambda \|\X\|_{\sigma_p}.$$
In particular, for $p=1$ one exactly recovers nuclear norm penalized matrix completion.
\end{proposition}

Let us briefly comment on the appearance of the nuclear norm in Corollary \ref{cor:robMC} and Proposition \ref{prop:robMC}. In light of Remark \ref{remark:ell1ComesFromEll0}, it is not surprising that such a penalty can be derived by working directly with the rank function (nuclear norm is the convex envelope of the rank function on the ball $\{\X:\|\X\|_{\sigma_\infty}\leq 1\}$, which is why the nuclear norm is typically used to replace rank \cite{fazelThesis,pablo}). We detail this argument as before. For any $p\in[1,\infty]$ and $\Gamma = \{\X\in\R^{m\times n}: \|\X\|_{\sigma_p}\leq 1\}$, one can show that
\begin{equation}\label{eqn:uncRepRank}
\U_{\ind(\sigma_1,\p(F_2))} = \left\{ \D\text{ linear}: \max_{\X \in \Gamma} \frac{\|\D(\X)\|_{\p(F_2)}}{\rank(\X)}\leq \lambda\right\}.
\end{equation}
Therefore, similar to the vector case with an underlying $\ell_0$ penalty which becomes a Lasso $\ell_1$ penalty, rank leads to the nuclear norm from the robustification setting without directly invoking convexity.

\subsection{Robust PCA}

We now turn our attention to the implications of Theorem \ref{thm:genMatrix} for PCA. We begin by noting robust analogues of $\min_{\X} \|\Y-\X\|$ under the $F_2$ and $\sigma_\infty$ norms. This is distinct from the considerations in \cite{xuChap} on robustness of PCA with respect to training and testing sets.
\begin{corollary}\label{cor:PCA}
For any $p\in[1,\infty]$ one has that
$$\min_{\X} \max_{\D\in\U_{\ind( \sigma_p,F_2)}} \|\Y-\X-\D(\X)\|_{F_2} = \min_{\X} \|\Y-\X\|_{F_2} + \lambda \|\X\|_{\sigma_p}$$
and
$$\min_{\X} \max_{\D\in\U_{\ind(\sigma_p,\sigma_\infty)}} \|\Y-\X-\D(\X)\|_{\sigma_\infty} = \min_{\X} \|\Y-\X\|_{\sigma_\infty} + \lambda \|\X\|_{\sigma_p}.$$
\end{corollary}

We continue by considering robust PCA as presented in \cite{candesrpca}. Suppose that $\U$ is some collection of linear maps $\D:\R^{m\times n}\to\R^{m\times n}$ and $\|\cdot\|$ is some norm so that for any $\Y,\X\in\R^{m\times n}$
$$\max_{\D\in\U} \|\Y-\X-\D(\X)\| = \|\Y-\X\|_{F_1}+\lambda \|\X\|_{\sigma_1}.$$
It is easy to see that this implies $\|\cdot\|=\|\cdot\|_{F_1}$. These observations, combined with Theorem \ref{thm:genMatrix}, imply the following:
\begin{proposition}\label{prop:robPCAnotUncertainPCA}
The problem \eqref{eqn:RobPCA} can be written as an uncertain version of $\min_{\X} \|\Y-\X\|$ subject to additive, linear uncertainty in $\X$ if and only if $\|\cdot\|$ is the $1$-Frobenius norm $F_1$. In particular, \eqref{eqn:RobPCA} does not arise as uncertain versions of PCA (using $F_2$ or $\sigma_\infty$) under such a model of uncertainty.
\end{proposition}

This result is not entirely surprising. This is because robust PCA attempts to solve, based on its model of $\Y=\X+\E$ where $\X$ is low-rank and $\E$ is sparse, a problem of the form
$$\min_{\X} \|\Y-\X\|_{F_0}+ \lambda \rank (\X) ,$$
where $\|\mb A\|_{F_0}$ is the number of nonzero entries of $\mb A$. In the usual way, $F_0$ and $\rank$ are replaced with surrogates $F_1$ and $\sigma_1$, respectively. Hence, \eqref{eqn:RobPCA} appears as a convex, regularized form of the problem
$$\begin{array}{ll}
\ds\min_{\X}&\|\Y-\X\|_{F_1}\\
\st&\rank(\X)\leq k.
\end{array}$$

Again, as with matrix completion, it is possible to show that \eqref{eqn:RobPCA} and uncertain forms of PCA with a nuclear norm penalty (as appearing in Corollary \ref{cor:PCA}) can be derived using the true choice of penalizer, rank, instead of imposing an \emph{a priori} assumption of a nuclear norm penalty. We summarize this, without proof, as follows:

\begin{proposition}\label{prop:pcaRankGivesNuclearNorm}
For any $p\in[1,\infty]$ and any norm $\|\cdot\|$,
$$\min_{\X\in\Gamma} \max_{\D\in\U_{\Gamma(\rank,\|\cdot\|)}}\|\Y-\X-\D(\X)\| = \min_{\X\in\Gamma} \|\Y-\X\| + \lambda \|\X\|_{\sigma_1},$$
where $\Gamma = \{\X\in\R^{m\times n}:\|\X\|_{\sigma_p}\leq1\}$ and
$$\U_{\Gamma\ind(\rank,\|\cdot\|)} = \left\{ \D\text{ linear}: \ds\max_{\X \in \Gamma} \frac{\|\D(\X)\|}{\rank(\X)}\leq \lambda\right\}.$$
\end{proposition}

\subsection{Non-equivalence of robustification and regularization}

As with vector regression it is not always the case that robustification is equivalent to regularization in matrix estimation problems. For completeness we provide analogues here of the linear regression results. We begin by stating results which follow over with essentially identical proofs from the vector case; proofs are not included here. Then we characterize precisely when another plausible model of uncertainty leads to equivalence.

We begin with the analogue of Proposition \ref{prop:alwaysUpperBound}.

\begin{proposition}\label{prop:alwaysUpperBoundMAT}
Let $\U\sub\{\text{linear maps }\D:\R^{m\times n}\to\R^{m\times n}\}$ be any non-empty, compact set and $g:\R^{m\times n}\to\R$ a seminorm. Then there exists some seminorm $\oh:\R^{m\times n}\to\R$ so that for any $\z,\X\in\R^{m\times n}$,
$$\max_{\D\in\U} g(\z+\D(\X)) \leq g(\z) + \oh(\X),$$
with equality when $\z=\mb0$.
\end{proposition}

As before with Theorem \ref{thm:discrepLbUb} and Propositions \ref{prop:nonequivRLSuniversal} and \ref{prop:rowWiseUncFullCharac}, one can now compute $\oh$ for a variety of problems.

\begin{proposition}\label{prop:discrepLbUbMAT}
For any $\z,\X\in\R^{m\times n}$,
\begin{align}
\|\z\|_{F_p} + \frac{\lambda}{\delta_{mn}(q,p)}\|\X\|_{F_{q^*}}&\leq \max_{\D\in\U_{F_q}} \|\z+\D(\X)\|_{F_p} \label{eqn:discMat}\\
&\leq \|\z\|_{F_p} + \lambda\delta_{mn}(p,q)\|\X\|_{F_{q^*}}\label{eqn:discMat_2}
\end{align}
where $\|\D\|_{F_q}$ is interpreted as the $F_q$ norm on the matrix representation of $\D$ in the standard basis. In particular, if $p\neq q$ and $p\in(1,\infty)$, then for any $\X\neq\mb0$ the upper bound in \eqref{eqn:discMat_2} is strict  for almost all $\z$ (so long as $mn\geq2$). Further, when $p\neq q$ and $p\in(1,\infty)$, the gap in the lower bound in \eqref{eqn:discMat} is arbitrarily small for all $\X$.
\end{proposition}

\begin{proposition}\label{prop:nonequivRLSuniversalMAT}
For any $\z,\X\in\R^{m\times n}$,
\begin{align}
\|\z\|_p + \frac{\lambda}{\delta_{mn}(2,p)}\|\X\|_{F_2}&\leq \max_{\D\in\U_{\sigma_q}} \|\z+\D(\X)\|_{F_p} \label{eqn:discMat2}\\
&\leq \|\z\|_{F_p} + \lambda\delta_{mn}(p,2)\|\X\|_{F_2}\label{eqn:discMat2_2}.
\end{align}
In particular, if $p\notin\{1, 2,\infty\}$, then for all $\X\neq\mb0$ the upper bound in \eqref{eqn:discMat2_2} is strict for almost all $\z$ (so long as $mn\geq2$). Further, if $p\notin\{1,2,\infty\}$, the gap in the lower bound in \eqref{eqn:discMat2} is arbitrarily small for all $\X$.
\end{proposition}

We now turn our attention to non-equivalencies which may arise under different models of uncertainty instead of the general matrix model of linear uncertainty which we have included here, where
$$[\D(\X)]_{ij} = \sum_{\ell k} \Delta_{\ell k}^{(ij)} X_{\ell k} = \langle \D^{(ij)},\X\rangle,$$
with $\D^{(ij)}\in\R^{m\times n}$. Another plausible model of uncertainty is one for which the $j$th column of $\D(\X)$ only depends on $\X_j$, the $j$th column of $\X$ (or, for example, with columns replaced by rows). We now examine such a model. In this setup, we now have $n$ matrices $\D^{(j)}\in\R^{m\times m}$ and we define the linear map $\D$ so that the $j$th column of $\D(\X)\in\R^{m\times n}$, denoted $[\D(\X)]_j$, is $[\D(\X)]_j := \D^{(j)}\X_j,$ which is simply matrix vector multiplication. Therefore,
\begin{equation}\label{eqn:newUncModel}
\D(\X) = \begin{bmatrix}
\D^{(1)}\X_1 &\cdots & \D^{(n)}\X_n
\end{bmatrix}.
\end{equation}

For an example of where such a model of uncertainty may arise, we consider matrix completion in the context of the Netflix problem. If one treats $\X_j$ as user $j$'s true ratings, then such a model addresses uncertainty within a given user's ratings, while not allowing uncertainty to have cross-user effects. This model of uncertainty does not rely on true matrix structure and therefore reduces to earlier results on non-equivalence in vector regression. As an example of such a reduction, we state the following proposition characterizing equivalence. Again, this is a direct modification of Theorem \ref{thm:discrepLbUb} and the proof we do not include here. 

\begin{proposition}
For the model of uncertainty in \eqref{eqn:newUncModel} with $\D^{(j)}\in\U_{F_{q_j}}$ for $j=1,\ldots,n$, where $q_j\in[1,\infty]$, one has for the problem $\ds\min_\X\max_{\D\in\U} \|\Y-\X-\D(\X)\|_{F_p}$ that $\oh$ is defined as
\begin{equation}\label{eqn:newUncModelUB}
\oh(\X) = \lambda \left(\sum_j \delta_m^p(p,q_j)\|\X_j\|_{q_j^*}^p\right)^{1/p}.
\end{equation}
Further, under such a model of uncertainty, robustification is equivalent to regularization with $\oh$ if and only if $p\in\{1,\infty\}$ or $p=q_j$ for all $j=1,\ldots,n$.
\end{proposition}

While the case of matrix regression offers a large variety of possible models of uncertainty, we see again as with vector regression that this variety inevitably leads to scenarios in which robustification is no longer directly equivalent to regularization. We summarize the conclusions of this section in Table \ref{tab:matRegEquiv}.

\begin{table}[h]
\centering
\begin{tabular}{|c|c|c|c|}\hline
Loss function & Uncertainty set & $\oh(\X)$ & Equivalence if and only if\\\hline\hline
seminorm $g$& $\U_{\ind(h,g)}$ ($h$ norm) & $\lambda h(\X)$ & always\\\hline
$F_p$&$\U_{\sigma_q}$&$\lambda\delta_{mn}(p,2)\|\X\|_{F_2}$ &  $p\in\{1,2,\infty\}$\\\hline
{$F_p$}&{$\U_{F_q}$}&{$\lambda \delta_{mn}(p,q)\|\X\|_{F_{q^*}}$}& $p\in \{1,q,\infty\}$\\\hline
\multirow{2}{*}{$F_p$}&$\U$ in \eqref{eqn:newUncModel} & \multirow{2}{*}{\eqref{eqn:newUncModelUB}}&($p=q_j\;\forall j$) or  \\
& with $\D^{(j)}\in\U_{F_{q_j}}$ & & $p\in\{1,\infty\}$\\\hline
\end{tabular}
\caption{Summary of equivalencies for robustification with uncertainty set $\U$ and regularization with penalty $\oh$, where $\oh$ is as given in Proposition \ref{prop:alwaysUpperBoundMAT}. Here by equivalence we mean that for all $\z,\X\in\R^{m\times n}$, $\max_{\D\in\U} g(\z+\X) = g(\z)+\oh(\X)$, where $g$ is the loss function, i.e., the upper bound $\oh$ is also a lower bound. Here $\delta_{mn}$ is as in Theorem \ref{thm:discrepLbUb}. Throughout $p,q\in[1,\infty]$ and $mn\geq2$.}
\label{tab:matRegEquiv}
\end{table}


\section{Conclusion}

In this work we have considered the robustification of a variety of problems from classical and modern statistical regression as subject to data uncertainty. We have taken care to emphasize that there is a fine line between this process of robustification and the usual process of regularization, and that the two are not always directly equivalent. While deepening this understanding we have also extended this connection to new domains, such as in matrix completion and PCA. In doing so, we have shown that the usual regularization approaches to modern statistical regression do not always coincide with an adversarial approach motivated by robust optimization.

\bibliographystyle{amsplain}
\bibliography{References}

\providecommand{\bysame}{\leavevmode\hbox to3em{\hrulefill}\thinspace}
\providecommand{\MR}{\relax\ifhmode\unskip\space\fi MR }
\providecommand{\MRhref}[2]{%
  \href{http://www.ams.org/mathscinet-getitem?mr=#1}{#2}
}
\providecommand{\href}[2]{#2}
\begin{thebibliography}{10}

\bibitem{bauschke}
H.H. Bauschke and P.L. Combettes, \emph{Convex analysis and monotone operator
  theory in {H}ilbert spaces}, Springer, 2011.

\bibitem{RObook}
A.~Ben-Tal, L.~El Ghaoui, and A.~Nemirovski, \emph{Robust optimization},
  Princeton University Press, 2009.

\bibitem{bta}
A.~Ben-Tal, E.~Hazan, T.~Koren, and S.~Mannor, \emph{Oracle-based robust
  optimization via online learning}, Operations Research \textbf{63} (2015),
  no.~3, 628--638.

\bibitem{ROsurvey}
D.~Bertsimas, D.B. Brown, and C.~Caramanis, \emph{Theory and applications of
  robust optimization}, SIAM Review \textbf{53} (2011), no.~3, 464--501.

\bibitem{ddunc1}
D.~Bertsimas, V.~Gupta, and N.~Kallus, \emph{Data-driven robust optimization},
  arXiv preprint arXiv:1401.0212 (2013).

\bibitem{bousquet}
O.~Bousquet, S.~Boucheron, and G.~Lugosi, \emph{Advanced lectures on machine
  learning}, ch.~Introduction to statistical learning theory, Springer, 2004.

\bibitem{boydvandenberghe}
S.~Boyd and L.~Vandenberghe, \emph{Convex optimization}, Cambridge University
  Press, 2004.

\bibitem{bradicFanWang}
J.~Bradic, J.~Fan, and W.~Wang, \emph{Penalized composite quasi-likelihood for
  ultrahigh dimensional variable selection}, Journal of the Royal Statistical
  Society, Series B \textbf{73} (2011), 325--349.

\bibitem{rechtmc}
E.~Cand\`{e}s and B.~Recht, \emph{Exact matrix completion via convex
  optimization}, Communications of the ACM \textbf{55} (2012), no.~6, 111--119.

\bibitem{candesrpca}
E.J. Cand\`es, X.~Li, Y.~Ma, and J.~Wright, \emph{Robust {P}rincipal
  {C}omponent {A}nalysis?}, Journal of the ACM \textbf{58} (2011), no.~3,
  11:1--37.

\bibitem{xuChap}
C.~Caramanis, S.~Mannor, and H.~Xu, \emph{Optimization for machine learning},
  ch.~Robust optimization in machine learning, MIT Press, 2011.

\bibitem{eivb}
R.J. Carroll, D.~Ruppert, L.A. Stefanski, and C.M. Crainiceanu,
  \emph{Measurement error in nonlinear models: A modern perspective}, 2nd ed.,
  CRC Press, 2006.

\bibitem{croux}
C.~Croux and A.~Ruiz-Gazen, \emph{High breakdown estimators for principal
  components: the projection-pursuit approach revisited}, Journal of
  Multivariate Analysis \textbf{95} (2005), 206--226.

\bibitem{demol}
C.~De~Mol, E.~De~Vito, and L.~Rosasco, \emph{Elastic-net regularization in
  learning theory}, Journal of Complexity \textbf{25} (2009), no.~2, 201--230.

\bibitem{eckart}
C.~Eckart and G.~Young, \emph{The approximation of one matrix by another of
  lower rank}, Psychometrika \textbf{1} (1936), 211--8.

\bibitem{fanFanBarut}
J.~Fan, Y.~Fan, and E.~Barut, \emph{Adaptive robust variable selection}, The
  Annals of Statistics \textbf{42} (2014), no.~1, 324--351.

\bibitem{fazelThesis}
M.~Fazel, \emph{Matrix rank minimization with applications}, Ph.D. thesis,
  Stanford University, 2002.

\bibitem{ghaoui}
L.~El Ghaoui and H.~Lebret, \emph{Robust solutions to least-squares problems
  with uncertain data}, SIAM Journal of Matrix Analysis and Applications
  \textbf{18} (1997), no.~4, 1035--64.

\bibitem{tls}
G.H. Golub and C.F.~Van Loan, \emph{An analysis of the total least squares
  problem}, SIAM Journal of Numerical Analysis \textbf{17} (1980), no.~6,
  883--893.

\bibitem{gan1}
I.J. Goodfellow, J.~Pouget-Abadie, M.~Mirza, B.~Xu, D.~Warde-Farley, Sh. Ozair,
  A.~Courville, and Y.~Bengio, \emph{Generative adversarial nets}, Advances in
  Neural Information Processing Systems 27, 2014, pp.~2672--2680.

\bibitem{gan2}
I.J. Goodfellow, J.~Shlens, and C.~Szegedy, \emph{Explaining and harnessing
  adversarial examples}, arXiv preprint arXiv:1412.6572 (2014).

\bibitem{hampel}
F.R. Hampel, \emph{The influence curve and its role in robust estimation},
  Journal of the American Statistical Association \textbf{69} (1974), 383--393.

\bibitem{hastie}
T.~Hastie, R.~Tibshirani, and J.~Friedman, \emph{The elements of statistical
  learning: data mining, inference, and prediction}, Springer, 2009.

\bibitem{hill}
R.W. Hill, \emph{Robust regression when there are outlires in the carriers},
  Ph.D. thesis, Harvard University, 1977.

\bibitem{hornjohnson}
R.A. Horn and C.R. Johnson, \emph{Matrix analysis}, second ed., Cambridge
  University Press, 2013.

\bibitem{huberMest}
P.J. Huber, \emph{Robust regression: asymptotics, conjectures and {M}onte
  {C}arlo}, The Annals of Statistics \textbf{1} (1973), 799--821.

\bibitem{huber}
P.J. Huber and E.M. Ronchetti, \emph{Robust statistics}, second ed., Wiley,
  2009.

\bibitem{roussSurvey}
M.~Hubert, P.J. Rousseeuw, and S.~Van Aelst, \emph{High-breakdown robust
  multivariate methods}, Statistical Science \textbf{23} (2008), no.~1,
  92--119.

\bibitem{hubert}
M.~Hubert, P.J. Rousseeuw, and K.~Van den Branden, \emph{{ROBPCA}: a new
  approach to robust principal components analysis}, Technometrics \textbf{47}
  (2005), 64--79.

\bibitem{tlsc}
A.~Kukush, I.~Markovsky, and S.~Van Huffel, \emph{Consistency of the structured
  total least squares estimator in a multivariate errors-in-variables model},
  Journal of Statistical Planning and Inference \textbf{133} (2005), 315--358.

\bibitem{al1}
A.S. Lewis, \emph{Robust regularization}, Tech. report, School of ORIE, Cornell
  University, 2002.

\bibitem{al2}
A.S. Lewis and C.H.J. Pang, \emph{Lipschitz behavior of the robust
  regularization}, SIAM Journal on Control and Optimization \textbf{48} (2009),
  no.~5, 3080--3104.

\bibitem{mallows}
C.L. Mallows, \emph{On some topics in robustness}, Tech. report, Bell
  Laboratories, 1975.

\bibitem{tlsoverview}
I.~Markovsky and S.~Van Huffel, \emph{Overview of total least-squares methods},
  Signal Processing \textbf{87} (2007), 2283--2302.

\bibitem{robStatsSurvey}
S.~Morgenthaler, \emph{A survey of robust statistics}, Statistical Methods and
  Applications \textbf{15} (2007), 271--293.

\bibitem{mosci}
S.~Mosci, L.~Rosasco, M.~Santoro, A.~Verri, and S.~Villa, \emph{Solving
  structured sparsity regularization with proximal methods}, Joint European
  Conference on Machine Learning and Knowledge Discovery in Databases,
  Springer, 2010, pp.~418--433.

\bibitem{pablo}
B.~Recht, M.~Fazel, and P.A. Parrilo, \emph{Guaranteed minimum-rank solutions
  of linear matrix equations via nuclear norm minimization}, SIAM Review
  \textbf{52} (2010), no.~3, 471--501.

\bibitem{rousseeuw}
P.J. Rousseeuw, \emph{Least median of squares regression}, Journal of the
  American Statistical Association \textbf{79} (1984), 871--80.

\bibitem{robRegBook}
P.J. Rousseeuw and A.M. Leroy, \emph{Robust regression and outlier detection},
  Wiley, 1987.

\bibitem{salibian}
M.~Salibian-Barrera, S.~Van Aelst, and G.~Willems, \emph{{PCA} based on
  multivariate {MM}-estimators with fast and robust bootstrap}, Journal of the
  American Statistical Association \textbf{101} (2005), no.~475, 1198--1211.

\bibitem{gan3}
U.~Shaham, Y.~Yamada, and S.~Negahban, \emph{Understanding adversarial
  training: Increasing local stability of neural nets through robust
  optimization}, arXiv preprint arXiv:1511.05432 (2015).

\bibitem{netflix}
SIGKDD and Netflix, \emph{Soft modelling by latent variables: the nonlinear
  iterative partial least squares ({NIPALS}) approach}, Proceedings of the KDD
  Cup and Workshop (2007).

\bibitem{tibshirani}
R.~Tibshirani, \emph{Regression shrinkage and selection via the {L}asso},
  Journal of the Royal Statistical Society, Series B \textbf{58} (1996),
  267--288.

\bibitem{ddunc2}
T.~Tulabandhula and C.~Rudin, \emph{Robust optimization using machine learning
  for uncertainty sets}, arXiv preprint arXiv:1407.1097 (2014).

\bibitem{xu}
H.~Xu, C.~Caramanis, and S.~Mannor, \emph{Robust regression and {L}asso}, IEEE
  Transactions in Information Theory \textbf{56} (2010), no.~7, 3561--74.

\bibitem{zhel}
H.~Zou and T.~Hastie, \emph{Regularization and variable selection via the
  elastic net}, Journal of the Royal Statistical Society: Series B \textbf{67}
  (2005), no.~2, 301--320.

\end{thebibliography}

\newpage

\section*{Appendix A} 

This appendix contains proofs and additional technical results for the vector regression setting. We prove our results in the vector setting, from which the primary results on matrices follow as a direct corollary. 

\begin{proof}[Proof of Theorem \ref{thm:discrepLbUb}]

\begin{enumerate}[(a)]

\item 

We begin by proving the upper bound. Here we proceed by showing that $\oh$ above is precisely $\oh(\bb) = \lambda\delta_m(p,q)\|\bb\|_{q^*}$. Now observe that for any $\D\in\U_{F_q}$,
\begin{equation}\label{eqn:midProof}
\|\D\bb\|_p \leq \delta_m(p,q) \|\D\bb\|_q \leq \delta_m(p,q) \|\D\|_{F_q} \|\bb\|_{q^*} \leq \delta_m(p,q)\lambda \|\bb\|_{q^*}.
\end{equation}
The first inequality follows by the definition of the discrepancy function $\delta_m$. The second inequality follows from a well-known matrix inequality: $\|\D\bb\|_{q}\leq \|\D\|_{F_q}\|\bb\|_{q^*}$ (this follows from a simple application of H\"older's inequality). Now observe that in the chain of inequalities in \eqref{eqn:midProof}, if one takes any $\uu\in\argmax \delta_m(p,q)$ and any $\vv\in\argmax_{\|\vv\|_q=1} \vv'\bb$, then $\hd:=\lambda\uu\vv'\in\U_{F_q}$ and $\|\hd\bb\|_p = \delta_m(p,q)\lambda \|\bb\|_{q^*}$. Hence, $\oh (\bb) = \delta_m(p,q)\lambda \|\bb\|_{q^*}$. This proves the upper bound.

\item We now prove that for $p\in\{1,\infty\}$ that one has equality for all $(\zz,\bb)\in\R^{m}\times\R^n$. This follows an argument similar to that needed for Theorem \ref{thm:genMatrix}. First consider the case when $p=1$. Fix $\zz\in\R^m$. Again let $\uu\in\argmax\delta_m(1,q)$ and $\vv\in\argmax_{\|\vv\|_q=1}\vv'\bb$. Without loss of generality we may assume that $\sign(z_i) = \sign(u_i)$ for $i=1,\ldots,m$ (one may change the sign of entries of $\uu$ and it is still in $\argmax \delta_m(1,q)$). Then again we have $\hd:=\lambda \uu\vv'\in\U_{F_q}$ and 
\begin{align*}
\|\zz+\hd\bb\|_1 &= \|\zz+ \lambda \uu\vv'\bb\|_1 = \|\zz+\lambda \|\bb\|_{q^*}\uu\|_1\\
& = \|\zz\|_1 + \lambda\|\bb\|_{q^*}\|\uu\|_1 = \|\zz\|_1 + \lambda\|\bb\|_{q^*}\delta_m(1,q).
\end{align*}
Hence, one has equality in the upper bound for $p=1$, as claimed.

We now turn our attention to the case $p=\infty$. Note that $\delta_m(\infty,q) = 1$ because $\|\zz\|_\infty\leq \|\zz\|_q$ for all $\zz\in\R^m$. Fix $\zz\in\R^m$, and again let $\vv\in\argmax_{\|\vv\|_q=1}\vv'\bb$. Let $\ell\in\{1,\ldots,m\}$ so that $|z_\ell| = \|\zz\|_\infty$. Define $\uu=\sign(z_\ell)\mb e_\ell\in\R^m$, where $\mb e_\ell$ is the vector whose only nonzero entry is a $1$ in the $\ell$th position. Now observe that $\hd:=\lambda \uu\vv'\in\U_{F_q}$ and
\begin{align*}
\|\zz+\hd\bb\|_\infty &= \|\zz+ \sign(z_\ell)\lambda \|\bb\|_{q^*}\mb e_\ell\|_\infty\\
& = \|\zz\|_\infty + \lambda\|\bb\|_{q^*} \|\mb e_\ell\|_\infty = \|\zz\|_\infty + \lambda\|\bb\|_{q^*},
\end{align*}
which proves equality in \eqref{eqn:discUB}, as was to be shown.

\item To proceed, we examine the case where $p\in(1,\infty)$ and consider for which $(\zz,\bb)$ the inequality in \eqref{eqn:discUB} is strict. Fix $\bb\neq\mb0$. For $p\in(1,\infty)$ and $\y,\zz\in\R^m$, one has by Minkowski's inequality that $\|\y+\zz\|_p = \|\y\|_p+\|\zz\|_p$ if and only if one of $\y$ or $\zz$ is a non-negative scalar multiple of the other. To have equality in \eqref{eqn:discUB}, it must be that there exists some $\D\in\argmax_{\D\in\U_{F_q}} \|\D\bb\|_p$ for which $\|\zz+\D\bb\|_p = \|\zz\|_p + \|\D\bb\|_p$. For any $\zz\neq\mb0$ this observation, combined with Minkowski's inequality, implies that
$$\|\D\|_{F_q} = \lambda,\;\D\bb = \mu\zz\;\text{ for some }\mu\geq0,\;\text{ and } \|\D\bb\|_p = \lambda\delta_m(p,q)\|\bb\|_{q^*}.$$
The first and last equalities imply that $\D\bb\in\lambda\|\bb\|_{q^*}\argmax\delta_m(p,q)$. Note that $\argmax\delta_m(p,q)$ is finite whenever $p\neq q$ and $m\geq2$, a geometric property of $\ell_p$ balls. Hence, taking any $\zz$ which is not a scalar multiple of a point in $\argmax\delta_m(p,q)$  implies by Minkowski's inequality that
$$\max_{\D\in\U_{F_q}} \|\zz+\D\bb\|_p < \|\zz\|_p+\lambda \delta_m(p,q)\|\mb \bb\|_{q^*}.$$
Hence, for any $\bb\neq\mb0$, the inequality in \eqref{eqn:discUB} is strict for all $\zz$ not in a finite union of one-dimensional subspaces, so long as $p\in(1,\infty)$, $p\neq q$, and $m\geq 2$.

\item We now prove the lower bound in \eqref{eqn:discLB}. If $\zz=\mb0$ then there is nothing to show, and therefore we assume $\zz\neq\mb0$. Let $\vv\in\R^n$ so that
$$\vv\in\argmax_{\|\vv\|_q=1} \vv'\bb.$$
Hence $\vv'\bb=\|\bb\|_{q^*}$ by the definition of the dual norm. Define $\hd = \frac{\lambda }{\|\zz\|_q} \zz\vv'$. Observe that $\hd\in\U_{F_q}$. Further, note that $\|\zz\|_q \leq \delta_m(q,p)\|\zz\|_p$ by definition of $\delta_m$ and therefore $1/\delta_m(q,p)\leq \|\zz\|_p/\|\zz\|_q.$ Putting things together,
\begin{align*}
\|\zz\|_p + \frac{\lambda\|\bb\|_{q^*}}{\delta_m(q,p)} &\leq \|\zz\|_p + \frac{\lambda \|\zz\|_p\|\bb\|_{q^*}}{\|\zz\|_q} \\
&= \|\zz\|_p \left(1+ \frac{\lambda \|\bb\|_{q^*}}{\|\zz\|_q}\right) \\
&= \|\zz+\hd \bb\|_p \\
&\leq \max_{\D\in\U_{F_q}} \|\zz+\D\bb\|_p.
\end{align*}
This completes the proof of the lower bound.

\item To conclude we prove that the gap in \eqref{eqn:discLB} can be made arbitrarily small for $p\in(1,\infty)$. We proceed in several steps. We first prove that for any $\zz\neq\mb 0$ that
\begin{equation}\label{eqn:subClaim1}\lim_{\alpha\to\infty} \left(\max_{\D\in\U_{F_q}} \|\alpha\zz+\D\bb\|_p - \|\alpha\zz\|_p \right) = \frac{\lambda\|\bb\|_{q^*}\|\zz^{p-1}\|_{q^*}}{\|\zz\|_p^{p-1}},
\end{equation}
where we use the shorthand $\zz^{p-1}$ to denote the vector in $\R^m$ whose $i$th entry is $|z_i|^{p-1}$. Observe that
$$\max_{\D\in\U_{F_q}} \|\alpha\zz+\D\bb\|_p = \max_{\|\uu\|_q \leq \lambda \|\bb\|_{q^*}}\|\alpha\zz+\uu\|_p.$$
It is easy to argue that we may assume without any loss of generality that $\uu\in\argmax_{\|\uu\|_q\leq\lambda\|\bb\|_{q^*}} \|\alpha\zz+\uu\|_p$ has $\sign(u_i) = \sign(\alpha z_i)$, where
$$\sign(a) = \left\{
\begin{array}{ll}
1, &a \geq0\\
-1,& a < 0.
\end{array}\right.$$
Therefore, we restrict our attention to $\zz\geq\mb0$, $\zz\neq\mb0$, and $\uu\geq\mb0$. For any $\uu$ such that $\|\uu\|_q\leq \lambda\|\bb\|_{q^*}$ and $\uu\geq \mb0$, note that
\begin{align*}
\lim_{\alpha\to\infty} \|\alpha\zz+\uu\|_p-\|\alpha\zz\|_p &=  \lim_{\alpha\to\infty} \frac{\|\zz+\uu/\alpha\|_p - \|\zz\|_p}{1/\alpha}\\
&= \lim_{\bar\alpha\to0^+} \frac{\|\zz+\bar\alpha\uu\|_p - \|\zz\|_p}{\bar\alpha}\\
&= \left.\frac{d}{d\bar\alpha}\right|_{\bar\alpha=0}\|\zz+\bar\alpha\uu\|_p\\
&= \frac{\uu'\zz^{p-1}}{\|\zz\|_p^{p-1}}.
\end{align*}
We can now proceed to finish the claim in \eqref{eqn:subClaim1} (still restricting attention to $\zz\geq \mb0$ without loss of generality). By the above arguments, for any $\uu\geq\mb0$ and any $\epsilon>0$ there exists some $\hat\alpha=\hat\alpha(\uu)>0$ sufficiently large so that for all $\alpha>\hat\alpha$,
$$\left|\|\alpha\zz+\uu\|_p - \|\alpha\zz\|_p - \frac{\uu'\zz^{p-1}}{\|\zz\|_p^{p-1}}\right|\leq\epsilon.$$
It remains to be shown that for any $\epsilon>0$ there exists some $\hat\alpha$ so that for all $\alpha>\hat \alpha$,
$$\left|\left(\max_{\|\uu\|_q\leq \lambda\|\bb\|_{q^*}}\|\alpha\zz+\uu\|_p-\|\alpha\zz\|_p\right) - \left(\max_{\|\uu\|_q\leq \lambda\|\bb\|_{q^*}} \frac{\uu'\zz^{p-1}}{\|\zz\|_p^{p-1}}\right)\right|\leq \epsilon.$$

We prove this as follows. Let $\epsilon>0$. Choose points $\{\uu_1,\ldots,\uu_M\}\sub\R^m$ with $\|\uu_j\|_q=\lambda\|\bb\|_{q^*}\;\forall j$ so that for any $\uu\in\R^m$ with $\|\uu\|_q = \lambda\|\bb\|_{q^*}$, there exists some $j$ so that $\|\uu-\uu_j\|_p\leq \epsilon/3$ (note that our choice of $\ell_p$ here is intentional). Now observe that for any $\alpha$,
\begin{align*}
\max_j\|\alpha\zz+\uu_j\|_p &\leq \max_{\|\uu\|_q\leq \lambda\|\bb\|_{q^*}} \|\alpha\zz+\uu\|_p\\
&\leq \max_j \left(\max_{\|\uu-\uu_j\|_p\leq \epsilon/3} \|\alpha\zz+\uu\|_p\right)\\
&=\max_j \left(\max_{\|\bar\uu\|_p\leq \epsilon/3} \|\alpha\zz+\uu_j+\bar\uu\|_p\right)\\
&\leq\max_j \left(\max_{\|\bar\uu\|_p\leq \epsilon/3} \|\alpha\zz+\uu_j\|_p+\|\bar\uu\|_p\right)\\
&= \epsilon/3+\max_j\|\alpha\zz+\uu_j\|_p.
\end{align*}
Similarly, one has for $\bzz = \zz^{p-1}/\|\zz\|_p^{p-1}$ that $\left|\max_j \uu_j'\bzz - \max_{\|\uu\|_q\leq \lambda \|\bb\|_{q^*}} \uu'\bzz\right|\leq\epsilon/3$. (This uses the fact that $\|\bzz\|_{p^*}=1$.)
Now for each $j$ choose $\hat\alpha_j$ so that for all $\alpha>\hat\alpha_j$,
$$\left|\|\alpha\zz+\uu_j\|_p - \|\alpha \zz\|_p - \uu_j'\bzz\right|\leq \epsilon/3.$$
Define $\hat\alpha = \max_j\hat\alpha_j$. Now observe that by combining the above two observations, one has for any $\alpha>\hat\alpha$ that
\begin{align*}
\left|\left(\max_{\|\uu\|_q\leq\lambda\|\bb\|_{q^*}}\right.\right. &\|\alpha\zz+\uu\|_p-\|\alpha\zz\|_p\bigg) - \left.\left(\max_{\|\uu\|_q\leq \lambda\|\bb\|_{q^*}} \uu'\bzz\right)\right|\leq \\
&\leq 2\epsilon/3 + \left|\left(\max_j\|\alpha\zz+\uu_j\|_p-\|\alpha\zz\|_p\right) - \left(\max_\ell \uu_\ell'\bzz\right)\right|\\
&\leq 2\epsilon/3+\max_j \left|\|\alpha\zz+\uu_j\|_p-\|\alpha\zz\|_p - \uu_j'\bzz\right|\\
&\leq 2\epsilon/3+\epsilon/3=\epsilon.
\end{align*}
Noting that $\max_{\|\uu\|_q\leq\lambda\|\bb\|_{q^*}}\uu'\bzz = \lambda\|\bb\|_{q^*}\|\bzz\|_{q^*}$ concludes the proof of \eqref{eqn:subClaim1}. We now claim that
\begin{equation}\label{eqn:subClaim2}
\min_{\zz}\frac{\|\zz^{p-1}\|_{q^*}}{\|\zz\|_p^{p-1}} = \frac{1}{\delta_m(q,p)}.
\end{equation}
First note that
\begin{equation}\label{eqn:subClaim3}
\min_{\zz}\frac{\|\zz^{p-1}\|_{q^*}}{\|\zz\|_p^{p-1}}= \min_{\tzz} \frac{\|\tzz\|_{q^*}}{\|\tzz\|_{p^*}}.
\end{equation}
We prove this as follows: given $\zz$, let $\tzz=\zz^{p-1}$. Then one can show that $\|\tzz\|_{p^*}/\|\zz\|_p^{p-1} = 1,$ and so $\|\tzz\|_{p^*}/\|\tzz\|_{q^*}=\|\zz\|_p^{p-1}/\|\zz^{p-1}\|_{q^*}.$ The converse is similar, proving \eqref{eqn:subClaim3}. Finally, note that
$$\min_{\tzz}\frac{\|\tzz\|_{q^*}}{\|\tzz\|_{p^*}} = \frac{1}{\delta_m(p^*,q^*)}$$
which follows from an elementary analysis using the definition of $\delta_m$. Combined with the observation that $\delta_m(p^*,q^*)=\delta_m(q,p)$, which follows by a simply duality argument (or by inspecting the formula), we have that \eqref{eqn:subClaim2} is proven. To finish the argument, pick any $\zz\in\ds\argmin_{\zz}\|\zz^{p-1}\|_{q^*}/\|\zz\|_p^{p-1}.$ Per \eqref{eqn:subClaim2}, $\|\zz^{p-1}\|_{q^*}/\|\zz\|_p^{p-1} = 1/\delta_m(q,p)$. Hence, now applying \eqref{eqn:subClaim1}, given any $\epsilon>0$, there exists some $\alpha>0$ large enough so that
$$\left|\left(\max_{\D\in\U_{F_q}} \|\alpha\zz+\D\bb\|_p \right) - \left(\|\alpha\zz\|_p + \frac{\lambda}{\delta_m(q,p)}\|\bb\|_{q^*}\right)\right|\leq\epsilon.$$
Therefore, the gap in the lower bound in \eqref{eqn:discLB} can be made arbitrarily small for any $\bb\in\R^n$. This concludes the proof.
\end{enumerate}
\end{proof}

\section*{Appendix B}

This appendix includes an example of choice of loss function and uncertainty set under which (a) regularization is not equivalent to robustification in general and (b) there exist problem instances for which the regularization path and robustification path have no intersection. The example we give is in the vector setting for simplicity, although the generalization to matrices is obvious.

In particular, let $m=m$ and $n=1$, \text{and} consider $\U = \U_{(1,2)}$ and loss function $\ell_2$, with $\y = \begin{pmatrix}
1\\1\\1
\end{pmatrix}$ and $\X = \begin{pmatrix}
2\\1\\0
\end{pmatrix}$. In symbols, the problem of interest is
$$\min_{\beta} \max_{\D\in\U_{(1,1)}} \|\y-(\X+\D)\beta\|_2,$$
where we simply use $\beta$ because $\bb$ is $1$-dimensional. This can be rewritten exactly as
$$\min_\beta \max_{\substack{\uu:\\\|\uu\|_1 \leq \lambda |\beta|}} \left\|\begin{pmatrix}
1\\1\\1
\end{pmatrix} - \begin{pmatrix}
2\\1\\0
\end{pmatrix}\beta + \uu \right\|_2.$$
One can argue in an elementary way that the (unique) solution to this problem, regardless of $\lambda\in(0,\infty)$, is $\beta^*=3/5$. 

We now turn our attention to the corresponding regularization problem, which as per Corollary \ref{prop:rowWiseUncFullCharac}, is
$$\min_\beta \|\y-\X \beta\|_2 + \rho |\beta|.$$
By using a calculus argument, it is possible to show that as $\rho\in(0,\infty)$ varies, the set of all solutions $\beta^*$ is the open interval $(0,3/5)$. In other words, the solution to the regularization problem is never a solution to the robustification, and vice versa. This completes the counterexample.

\end{document}